\documentclass[12pt,a4paper,reqno]{amsart}
\usepackage[all]{xy}
\usepackage{amsmath,amssymb,amsfonts,amsthm}
\usepackage{verbatim} 
\usepackage{graphicx} 
\usepackage{harvard}
\usepackage{amsfonts}
\usepackage{amsthm}
\usepackage{amsmath}
\usepackage{amscd}
\usepackage{mathtools}
\usepackage[latin2]{inputenc}
\usepackage{t1enc}
\usepackage[mathscr]{eucal}
\usepackage{indentfirst}
\usepackage{graphicx}
\usepackage{graphics}
\usepackage{pict2e}
\usepackage{epic}
\usepackage{faktor}
\usepackage{xfrac}   
\numberwithin{equation}{section}
\usepackage[margin=2.9cm]{geometry}
\usepackage{epstopdf} 
\usepackage{cite}
\usepackage{hyperref}
\usepackage{lipsum}

\theoremstyle{plain}
\newtheorem{Th}{Theorem}[section]
\newtheorem{Lemma}[Th]{Lemma}
\newtheorem{Cor}[Th]{Corollary}
\newtheorem{Prop}[Th]{Proposition}

 \theoremstyle{definition}

\newtheorem{?}[Th]{Problem}

\newcommand{\Addresses}{{
  \bigskip
  \footnotesize

  \textsc{Department of Mathematics, The University of Chicago, Chicago, IL 60615}\par\nopagebreak
  \textit{E-mail address}: \texttt{nth@uchicago.edu}
}}
\numberwithin{equation}{section}

\begin{document}

\title{K$\ddot{\text{A}}$HLER-RICCI FLOW AND CONFORMAL SUBMERSION}

\author[HOAN NGUYEN]{NGUYEN THE HOAN}

\begin{abstract} \textit{We study singularity formation of K$\ddot{\text{a}}$hler-Ricci flow on a K$\ddot{\text{a}}$hler manifold that admits a horizontally homothetic conformal submersion into another K$\ddot{\text{a}}$hler manifold. We will derive necessary and sufficient conditions for the preservation of horizontally homothetic conformal submersion along the flow and establish the formation of type I singularity together with a standard splitting of the Cheeger-Gromov limit. This generalizes the setup of Calabi symmetry that was discussed in \cite{1} and \cite{2}, thus gives new proofs for the results listed there. }
\end{abstract}
\maketitle
\section{Introduction}
In \cite{2}, Song and Weinkove showed that for a K$\ddot{\text{a}}$hler-Ricci flow on a Hirzebruch surface (a $\mathbb{P}^1$ bundle over a projective space)  which satisfies certain rotational symmetry conditions (Calabi symmetry), then either it will collapse the $\mathbb{P}^1$ fibers, shrink the base to a point or contract the exceptional divisor. In all of these cases, the manifold considered as a metric spaces induced by K$\ddot{\text{a}}$hler metrics along the flow is known to perform a Gromov-Hausdorff convergence.  

The singularity behaviors of the flow in these cases are also of particular interests. Since it was proved in \cite{8} that the flow could only exist only for finite time interval $[0,T)$, the singularity model can only be of type I or II. Naturally for a finite time existence K$\ddot{\text{a}}$hler-Ricci flow, type I singularity is expected. The complete answers in the discussed settings were given by Fong in \cite{1}, Song in \cite{10} and Zhu in \cite{9}, corresponding to each of these above cases. For example, in the fibers collapsing case, Fong showed that the Cheeger-Gromov limit solution $(M_\infty,g_\infty(t))$ attained under the standard high curvature point-picking and metric normalizing must have the universal cover splitting isometrically as the product of a flat factor with a Riemannian surface. Therefore, if a type II singularity is formed, an analysis based on the asymptotic behavior of Riemannian curvatures shows that the universal cover of Cheeger-Gromov limit is product of $\mathbb{C}^n$ with a cigar soliton, contradicting to Perelman's local non-collapsing theorem.

One can show by direct calculation that for the set up of Calabi symmetry by Fong in \cite{1}, the bundle map under the flow always satisfies the condition of a horizontally homothetic conformal submersion, with dilation factors $f(\rho, t)$ evolve under a heat-type equation. The purpose of this paper is to study finite time singularity formation of K$\ddot{\text{a}}$hler-Ricci flow on manifold which admits a conformal submersion to another manifold, from this singularity results in \cite{1} follow. 

Inspired by the set up of Calabi symmetry, we consider the following scenario:
Suppose that $\pi: (M^{n+1},g_0)\rightarrow (B^n,h)$ is a horizontally homothetic conformal submersion with totally geodesic fibers between closed K$\ddot{\text{a}}$hler manifolds. Suppose that $(M^{n+1},g(t))$ is a finite solution of K$\ddot{\text{a}}$hler-Ricci flow on $M$ such that the following conditions hold:
\begin{enumerate}
    \item $\omega_0=\omega_{g_0}$
    \item The limiting K$\ddot{\text{a}}$hler class satisfies:
    \begin{equation}{\label{initial}}
        [\omega_0]+Tc_1(K_M)=[\pi^*\omega_B]
    \end{equation}
    where $T$ is the finite maximal time existence of the flow. 
     \item With respect to metric $g_t$ on $M$ and $h$ on $B$, the bundle map $\pi$ is a horizontally homothetic conformal submersion with dilation $f_t: M\rightarrow \mathbb{R}>0$, i.e. for $X,Y \in TM$ we have:
    \begin{equation}
        g_t(\mathcal{H}X,\mathcal{H}Y)=f_th(\pi_*X,\pi_*Y)
    \end{equation}
\end{enumerate}
Our first theorem gives a necessary and sufficient conditions under which the third condition satisfies, meaning that the horizontally homothetic conformal property of $\pi$ is preserved along the flow:
\begin{Th}\label{preservation of submersion}
Suppose that $\pi: (M^{n+1},g_0)\to (B^n,h)$ is a horizontally homothetic conformal submersion of dilation $f_0$ with totally geodesic fibers. Consider the K$\ddot{\text{a}}$hler-Ricci flow on $M$:
\[
    \dfrac{\partial}{\partial t}\omega=-\text{Ric}(\omega),\quad \omega(0)=\omega(g_0)
\]
Then the submersion $\pi$ remains being horizontally homothetic conformal submersion with totally geodesic fibers with respect to $(M,g(t))$ and $(B,h)$ if and only if the following conditions hold:
\begin{enumerate}
    \item $(B,h)$ is K$\ddot{\text{a}}$hler-Einstein manifold.
    \item The heat equation:
\[\dfrac{\partial f}{\partial t}=\Delta_t f-\dfrac{R_B}{n},\quad f(p,0)=f_0(p)\]
has solution $f(.,t)$ with vertical gradient for all time.
\end{enumerate}
\end{Th}
For the singularity analysis, we reconstruct the following $\textit{splitting lemma}$ of Fong in \cite{1}.  
\begin{Th}\label{main theorem}
Under the above scenario, we choose a sequence $(x_i, t_i) \in M\times [0,T)$ such that $t_i \rightarrow T$ and:
$K_i=||Rm(x_i)||_{g(t_i)} \rightarrow \infty $ as $i \rightarrow \infty$ and define the corresponding sequence of dilated metrics: 
\[g_i(t):=K_ig(t_i+K_i^{-1}t), t \in [-\beta_i, \alpha_i] \]
with $\beta_i\rightarrow \infty$ and $\alpha_i\rightarrow A\geq 0$. 
Suppose that the curvature tensors of $g_i(t)$ are uniformly bounded for all $i$, i.e there exists constant $C>0$ such that:
\[\sup_{M\times [-\beta_i,\alpha_i]}||Rm||_{g_i(t)}\leq C \]
Then, after passing through a sub-sequence, $(M, g_i(t), x_i)$ converge smoothly in the pointed Cheeger-Gromov sense to a complete ancient K$\ddot{\text{a}}$hler-Ricci solution \[(M_{\infty},g_{\infty}(t),x_{\infty})\] whose universal cover is of the form:
$(\mathbb{C}^n\times N_2, ||dz||^2\oplus h_2(t))$, with $n$ being the complex dimension of $B$. 
\end{Th}
From this theorem the formation of type I singularity follows. More precisely:
\begin{Th}\label{singularity}
If $\pi: (M^{n+1}, g_0)\to (B^n,h)$ satisfies the above conditions, then the flow encounters type I singularity. Furthermore, under the standard high curvature point picking $(x_i,t_i)$ with $K_i:=||Rm(x_i,t_i)||_{g(t_i)}=\max||Rm||_{g(t_i)}$ and $t_i\to T$ and:
\begin{equation}
    g_i(t)=K_ig(t_i+K_i^{-1}t), \quad t\in [-t_iK_i, (T-t_i)K_i)
\end{equation}
the pointed sequence $(M,g_i(t),x_i)$ converges smoothly after passing through a subsequence in the Cheeger-Gromov sense to an ancient $\kappa$-solution $(M_\infty, g_\infty(t),x_\infty)$, whose universal cover splits isometrically as:
\begin{equation}
    (\mathbb{C}^n\times \mathbb{CP}^1, ||dz||^2\oplus \omega_{FS}(t))
\end{equation}
\end{Th}
The organization of this paper is as follows: in section 2, we provide some backgrounds on conformal submersion. In section 3, we will prove that as long as the conformal submersion condition is preserved, the totally geodesic fibers property of $\pi$ will be preserved and derive a heat-type evolution equation for $f_t$. From this, we can further produce the necessary and sufficient conditions on initial metric such that the map $\pi$ remains being horizontally homothetic conformal submersion with totally geodesic fibers along the flow. Finally, we prove $\textbf{Theorem \ref{main theorem}}$ in a slightly more general setting, without assumption on the dimension of the manifolds. This introduce an alternative approach for $\textbf{Lemma 6.4}$ in \cite{1}. In the last section, the singularity model in $\textbf{Theorem \ref{singularity}}$ will be achieved. 

\textit{Acknowledgements.} The author would like to express sincere gratitude to Prof. Frederick Tsz-Ho Fong for his valuable advice and support during the project. 

\section{Conformal submersion}
In this section, we will present a brief introduction of conformal submersion between Riemannian manifolds (interested readers can refer to \cite{4}, \cite{7} and \cite{15} for more detailed information). Given a submersion $\pi: M\rightarrow B$ of Riemannian manifolds and $\pi(q)=p$. Let $\mathcal{V}_qM$ be the kernel of the tangent map $\pi_*$ at $q$. The orthogonal complement of $\mathcal{V}_qM$ with respect to $g_M$ is denoted by $\mathcal{H}_qM$. The map $\pi$ is called conformal submersion if there exists a function $f: M\rightarrow \mathbb{R}_{>0}$, such that for $u,v \in \mathcal{H}_qM$:
\[g_M(u,v)=f(q)g_B(\pi_*u,\pi_*v).\]

We denote by $\mathcal{H}_q$, $\mathcal{V}_q$ the horizontal and vertical distributions, respectively. 

Let $\mathcal{H}X$ (and $\mathcal{V}X$, respectively) be the projection of a vector field $X$ to the horizontal (vertical) distribution. We say that $\pi$ is \textit{horizontally homothetic conformal submersion} if :
\[\mathcal{H}(\nabla f)=0.\]

The O'Neill tensors corresponding to the conformal submersion are defined as follows:
\begin{equation}
    T_EF=\mathcal{V}\nabla_{\mathcal{V}E}\mathcal{H}F+\mathcal{H}\nabla_{\mathcal{V}E}\mathcal{V}F,
\end{equation}
\begin{equation}
    A_EF=\mathcal{V}\nabla_{\mathcal{H}E}\mathcal{H}F+\mathcal{H}\nabla_{\mathcal{H}E}\mathcal{V}F.
\end{equation}

It is obvious that the restriction of tensor $T$ to each fiber $\pi^{-1}(p)$ is the second fundamental form of that fiber. Therefore, the condition that the bundle map $\pi$ has totally geodesic fibers is equivalent to the vanishing of $T$ tensor. The following lemma in \cite{7} will be useful later:
\begin{Lemma}\label{A_1}
Suppose that $X, Y\in \mathcal{H}M$ are horizontal vector fields, then we have:
\begin{equation}
A_XY=\dfrac{1}{2}\left(\mathcal{V}[X,Y]-g(X,Y)\mathcal{V}\nabla\ln f\right).
\end{equation}
\end{Lemma}
The following result derived in \cite{15} gives a nice expression for Lie bracket of horizontal vector fields in the case of horizontally homothetic conformal submersion between K$\ddot{\text{a}}$hler manifolds. 
\begin{Lemma}\label{vertical of Lie}
Suppose that $\pi: (M,g)\rightarrow (B,h)$ is a horizontally homothetic conformal submersion of K$\ddot{\text{a}}$hler manifolds with dilation $f$ and totally geodesic fibers. Then for horizontal vector fields $X,Y$ on $TM$ we have:
\begin{equation}
    \mathcal{V}[X,Y]=\omega(X,Y)J(\nabla \ln f),
\end{equation}
where $\omega$ is the corresponding K$\ddot{\text{a}}$hler form on $M$. 
\end{Lemma}
\begin{proof}
Suppose that $\{V_1,\ldots,V_k\}$ is an orthonormal basis of the vertical distribution $\mathcal{V}M$. Then:
\begin{equation}\label{vertical1}
    \mathcal{V}[X,Y]=\sum_{i=1}^{k}g([X,Y],V_i)V_i.
\end{equation}
We have $g([X,Y],V_i)=-\omega([X,Y],JV_i)$. Using the fact that $\omega$ is K$\ddot{\text{a}}$hler, we have: $d\omega(X,Y,JV_i)=0$, which is equivalent to: 
\[X\omega(Y,JV_i)-Y\omega(X,JV_i)+JV_i\omega(X,Y)-\]
\[\omega([X,Y],JV_i)+\omega([X,JV_i],Y)-\omega([Y,JV_i],X)=0.\]
Since $\omega(X,JV_i)=\omega(Y,JV_i)=0$ and $[X,JV_i], [Y,JV_i]$ are vertical (to see this we may assume $X$ is basis, $\pi_*[X,JV_i]=[\pi_*X,\pi_*JV_i]=0)$, the identity reduces to:
\begin{equation}\label{closed omega}
    JV_i\omega(X,Y)=\omega([X,Y],JV_i).
\end{equation}
Denoted by $\omega_h$ the K$\ddot{\text{a}}$hler form of $B$, then by the condition of conformal submersion we have:
\begin{equation}
    JV_i\omega(X,Y)=JV_i(f\omega_h(\pi_*X,\pi_*Y))=\omega_h(\pi_*X,\pi_*Y)JV_if=\omega(X,Y)g(\nabla\ln f,JV_i).
\end{equation}
Since $g(\nabla\ln f,JV_i)=-g(J\nabla \ln f,V_i)$ we have:
\begin{equation}
    g([X,Y],V_i)=\omega(X,Y) g(J\nabla \ln f,V_i).
\end{equation}
This together with $(\ref{vertical1})$ leads to:
\begin{equation}
    \mathcal{V}[X,Y]=\sum_{i=1}^{k}\omega(X,Y)g(J\nabla\ln f,V_i)V_i=\omega(X,Y)J\nabla\ln f.
\end{equation}
We complete the proof. 
\end{proof}

\section{Preservation of conformal submersion}
In this section, we will prove our $\textbf{Theorem \ref{preservation of submersion}}$. 
\subsection{Necessary conditions} \quad 

Suppose that under the flow, the map $\pi: (M^{n+1},g(t)) \to (B,h)$ is horizontally homothetic conformal submersion of dilation $f_t$ with totally geodesic fibers at all time. To obtain some local properties along the flow, we fix a point $p\in M$ and a local submersion coordinate $(z_1,z_2,\ldots,z_n,\xi)$ in a neighborhood of $p$. 

We first obtain a necessary property on the form of the metrics along the flow:
\begin{Lemma}\label{ratio}
If the conformal submersion condition is preserved under the K$\ddot{\text{a}}$hler-Ricci flow, then for all $i=1,2,\ldots,n$ we must have:
\[
    \dfrac{g_{i\bar{\xi}}}{R_{i\bar{\xi}}}=\dfrac{g_{\xi\bar{\xi}}}{R_{\xi\bar{\xi}}},
\]
at all time $t$. In particular, that must happen at $t=0$.
\end{Lemma}
\begin{proof}
For any $X,Y$ we have:
\begin{equation}
    g(t)(\mathcal{H}X,\mathcal{H}Y)=f_t(p)h(\pi_*X,\pi_*Y).
\end{equation}\label{conformal}
Differentiating this with respect to $t$ and notice that $\dfrac{\partial}{\partial t}\mathcal{H}X=-\dfrac{\partial}{\partial t}\mathcal{V}X\in \mathcal{V}_pM$ we get:
\[
    R(\mathcal{H}X,\mathcal{H}Y)=-\dfrac{\partial f}{\partial t}h(\pi_*X,\pi_*Y).
\]
From this it follows that:
\begin{equation}\label{R_ij}
    R_{i\bar{j}}=-\dfrac{\partial f}{\partial t}h_{i\bar{j}}+\dfrac{R_{i\bar{\xi}}R_{\xi\bar{j}}}{R_{\xi\bar{\xi}}}.
\end{equation}
Notice that from conformal condition $(\ref{conformal})$ we also have similar expression for $g$ as:
\[
    g_{i\bar{j}}=fh_{i\bar{j}}+\dfrac{g_{i\bar{\xi}}g_{\xi\bar{j}}}{g_{\xi\bar{\xi}}}.
\]
Since $\dfrac{\partial}{\partial t}g_{i\bar{j}}=-R_{i\bar{j}}$ we must then have:
\[
\dfrac{\partial}{\partial t}\left(\dfrac{g_{i\bar{\xi}}g_{\xi\bar{j}}}{g_{\xi\bar{\xi}}}\right)=-\dfrac{R_{i\bar{\xi}}R_{\xi\bar{j}}}{R_{\xi\bar{\xi}}}.
\]
From this, we obtain $\dfrac{g_{i\bar{\xi}}}{R_{i\bar{\xi}}}=\dfrac{g_{\xi\bar{\xi}}}{R_{\xi\bar{\xi}}}$ by taking $i=j$.
\end{proof}

Notice that since $\dfrac{\partial}{\partial t}\omega=-\text{Ric}(\omega)$, the above relation can be rewritten as follows:
\[
    \dfrac{\partial}{\partial t}(\ln g_{i\bar{\xi}})=\dfrac{\partial}{\partial t}(\ln g_{\xi\bar{\xi}}) \iff \dfrac{\partial}{\partial t}\left(\dfrac{ g_{i\bar{\xi}}}{g_{\xi\bar{\xi}}}\right)=0.
\]

From this lemma, we can find time-independent smooth functions: 
\[s_1,s_2\ldots,s_n: U\subset M\to \mathbb{C},\] 
such that $g_{i\bar{\xi}}=s_ig_{\xi\bar{\xi}}$ and $R_{i\bar{\xi}}=s_iR_{\xi\bar{\xi}}$ for all $t$. Therefore, the matrix representations of K$\ddot{\text{a}}$hler metric and Ricci tensor are rewritten as follows:
\begin{equation}\label{matrix g}
g_{A\bar{B}}=
\left[
\begin{array}{c|c}
fh_{i\bar{j}}+ \dfrac{g_{i\bar{\xi}}g_{\xi\bar{j}}}{g_{\xi\bar{\xi}}}& s_ig_{\xi\bar{\xi}} \\
\hline
\overline{s_j}g_{\xi\bar{\xi}} & g_{\xi\bar{\xi}}
\end{array}
\right],
\end{equation}

\begin{equation}\label{ricci curvature}
\text{Ric}_{A\bar{B}}=
\left[
\begin{array}{c|c}
-\dfrac{\partial f}{\partial t}h_{i\bar{j}}+ \dfrac{R_{i\bar{\xi}}R_{\xi\bar{j}}}{R_{\xi\bar{\xi}}}& s_iR_{\xi\bar{\xi}} \\
\hline
\overline{s_j}R_{\xi\bar{\xi}} & R_{\xi\bar{\xi}}
\end{array}
\right].
\end{equation}

It is well-known that $\pi$ has totally geodesic fibers if and only if $T$ tensor vanishes. The following provides equivalent condition in term of the metrics:
\begin{Prop}
The map $\pi: (M, g(t))\to (B,h)$ has totally geodesic fibers ($T$ tensor vanishes), meaning that:
\[
    \Gamma^{i}_{\xi\xi}\big|=0, \quad \forall i=1,2,\ldots,n.
\]
If and only if $\dfrac{\partial}{\partial \bar{\xi}}s_i=0$ for all $i=1,2,\ldots,n$. Consequently, if initially $\pi$ has totally geodesic fibers and the horizontally homothetic condition is preserved, then it continues to have totally geodesic fibers at all time. 
\end{Prop}
\begin{proof}
The inverse matrix $[g]^{i\bar{j}}$ at any time $t$ can be calculated as:
\begin{equation}\label{matrix}
    g^{AB}=
    \begin{cases}
      \dfrac{1}{f}h^{i\bar{j}}, & \text{if}\ (A,B)=(i,\bar{j}) \\
      -\dfrac{1}{f}h^{i\bar{k}}\bar{s_k}, & \text{if}\ (A,B)=(i,\bar{\xi}) \\
      \dfrac{1}{g_{\xi\bar{\xi}}}+\dfrac{1}{f}h^{i\bar{j}}s_i\bar{s_j}, & \text{if}\ (A,B)=(\xi,\bar{\xi}).
    \end{cases}
  \end{equation}
 The Christoffel symbols can be calculated as follows:
 \begin{equation}
\begin{split}\label{Gamma}
 \Gamma^i_{\xi\xi}
 &=g^{i\bar{j}}\dfrac{\partial}{\partial\xi}g_{\xi\bar{j}}+g^{i\bar{\xi}}\dfrac{\partial}{\partial\xi}g_{\xi\bar{\xi}}\\
 &=\dfrac{1}{f}h^{i\bar{j}}\dfrac{\partial}{\partial\xi}\left(g_{\xi\bar{\xi}}\bar{s_j}\right)+g^{i\bar{\xi}}\dfrac{\partial}{\partial\xi}g_{\xi\bar{\xi}}\\
 &=\dfrac{1}{f}h^{i\bar{j}}\bar{s_j}\dfrac{\partial}{\partial\xi}g_{\xi\bar{\xi}}+\dfrac{1}{f}h^{i\bar{j}}\dfrac{\partial}{\partial\xi}\bar{s_j}g_{\xi\bar{\xi}}-\dfrac{1}{f}h^{i\bar{j}}\bar{s_j}\dfrac{\partial}{\partial\xi}g_{\xi\bar{\xi}}\\
 &=\dfrac{1}{f}h^{i\bar{j}}g_{\xi\bar{\xi}}\dfrac{\partial}{\partial\xi}\bar{s_j}.
\end{split}
\end{equation}
Therefore, the fibers of $\pi$ are totally geodesic if and only if for all $i=1,2,\ldots,n$:
\[
    h^{i\bar{j}}\dfrac{\partial}{\partial\xi}\bar{s_j}=0.
\]
This is equivalent to $[h]\big[\dfrac{\partial}{\partial\xi}\bar{s_1},\ldots,\dfrac{\partial}{\partial\xi}\bar{s_n}\big]^T=0$ and since $h$ is positive definite, it follows that $\dfrac{\partial}{\partial\xi}\bar{s_i}=0$ for all $i=1,2,\ldots,n$. 

As this condition is independent of time, once it holds at $t=0$, it is preserved at all time. 
\end{proof}
We can now derive the evolution equation for the dilation $f_t$:
\begin{Prop}
Suppose that $\pi: (M^{n+1},g(t))\to (B^n,h)$ with $g(t)$ is K$\ddot{\text{a}}$hler-Ricci flow on $M$ such that $\pi$ is homothetic conformal submersion with totally geodesic fibers at all time $t$, then the evolution equation of $f$ is as follows:
\begin{equation}\label{evolution of dilation}
    \dfrac{\partial f}{\partial t}=\Delta f-\dfrac{R^h}{n},
\end{equation}
where $R^h$ is the scalar curvature of $(B,h)$ at the projection point $\pi(p)$ on $B$. 
\end{Prop}
\begin{proof}
Along the K$\ddot{\text{a}}$hler-Ricci flow, the following evolution equation is well-known:
\[
    \dfrac{\partial}{\partial t}(\ln\det g)=\Delta(\ln\det g).
\]
Since $\det g=f^ng_{\xi\bar{\xi}}\det h$, the equation is equivalent to:
\begin{equation}\label{evolution}
    \Delta\ln g_{\xi\bar{\xi}}+n\Delta\ln f+\Delta\det h=\dfrac{\partial}{\partial t}\ln g_{\xi\bar{\xi}}+n\dfrac{\partial}{\partial t}\ln f.
\end{equation}
As before, we denote by $A,\bar{B}$ the indices corresponding to $z^1,z^2,\ldots, \xi$ and their conjugates and $i,\bar{j}$ the indices corresponding to $z^1,\ldots,z^n$ only. We can calculate each term of $(\ref{evolution})$ as follows:
\begin{flalign}
    \dfrac{\partial}{\partial t}\ln g_{\xi\bar{\xi}}
    &=-\dfrac{1}{g_{\xi\bar{\xi}}}R_{\xi\bar{\xi}}
    =\dfrac{1}{g_{\xi\bar{\xi}}}\dfrac{\partial^2}{\partial\xi\partial\bar{\xi}}\left(\ln g_{\xi\bar{\xi}}+n\ln f\right)&&\\\
    &=\dfrac{n}{fg_{\xi\bar{\xi}}}\dfrac{\partial^2f}{\partial\xi\partial\bar{\xi}}-\dfrac{n}{f^2g_{\xi\bar{\xi}}}\dfrac{\partial f}{\partial\xi}\dfrac{\partial f}{\partial\bar{\xi}}+\dfrac{1}{g_{\xi\bar{\xi}}^2}\dfrac{\partial^2g_{\xi\bar{\xi}}}{\partial\xi\partial\bar{\xi}}-\dfrac{1}{g_{\xi\bar{\xi}}^3}\dfrac{\partial g_{\xi\bar{\xi}}}{\partial\xi}\dfrac{\partial g_{\xi\bar{\xi}}}{\partial\bar{\xi}}&&
\end{flalign}

\begin{flalign}\label{term 2}
    n\Delta \ln f &=\dfrac{n\Delta f}{f}-\dfrac{n}{f^2}g^{A\bar{B}}\dfrac{\partial f}{\partial A}\dfrac{\partial f}{\partial \bar{B}}
    =\dfrac{n\Delta f}{f}-
    \dfrac{n}{f^2g_{\xi\bar{\xi}}}\dfrac{\partial f}{\partial\xi}\dfrac{\partial f}{\partial\bar{\xi}}&&
\end{flalign}
Here the second identity comes from the assumption that the gradient of $f$ is vertical during the flow. To calculate the first term, we will need the following identity which comes from the K$\ddot{\text{a}}$hler property of $g$: $\dfrac{\partial}{\partial z^i}g_{\xi\bar{\xi}}=\dfrac{\partial}{\partial \xi}g_{i\bar{\xi}}$:
\[
    \dfrac{\partial}{\partial z^i}g_{\xi\bar{\xi}}=\dfrac{\partial s_i}{\partial\xi}g_{\xi\bar{\xi}}+s_i\dfrac{\partial}{\partial\xi}g_{\xi\bar{\xi}}.
\]
From this we can calculate the following:
\begin{flalign}\label{term 1}
    \Delta\ln g_{\xi\bar{\xi}}
    &=\dfrac{1}{g_{\xi\bar{\xi}}}g^{A\bar{B}}\dfrac{\partial^2g_{\xi\bar{\xi}}}{\partial A\partial\bar{B}}-\dfrac{1}{g_{\xi\bar{\xi}}^2}g^{A\bar{B}}\dfrac{\partial g_{\xi\bar{\xi}}}{\partial A}\dfrac{\partial g_{\xi\bar{\xi}}}{\partial\bar{B}}&&\\
    &=\dfrac{1}{g_{\xi\bar{\xi}}}\left(\dfrac{1}{g_{\xi\bar{\xi}}}\dfrac{\partial^2g_{\xi\bar{\xi}}}{\partial\xi\partial\bar{\xi}}+g^{i\bar{j}}g_{\xi\bar{\xi}}\dfrac{\partial^2\bar{s_j}}{\partial z^i\partial\bar{\xi}}+g^{i\bar{j}}\dfrac{\partial\bar{s_j}}{\partial z^i}\dfrac{\partial g_{\xi\bar{\xi}}}{\partial\bar{\xi}}+g^{i\bar{j}}g_{\xi\bar{\xi}}\dfrac{\partial s_i}{\partial\xi}\dfrac{\partial\bar{s_j}}{\partial \bar{\xi}} \right)&&\\
    &\quad-\dfrac{1}{g_{\xi\bar{\xi}}^2}\left(\dfrac{1}{g_{\xi\bar{\xi}}}\dfrac{\partial g_{\xi\bar{\xi}}}{\partial\xi}\dfrac{\partial g_{\xi\bar{\xi}}}{\partial\bar{\xi}}+g^{i\bar{j}}g_{\xi\bar{\xi}}\dfrac{\partial s_i}{\partial \xi}\dfrac{\partial\bar{s_j}}{\partial\bar{\xi}} \right)
\end{flalign}
Referring all these calculations back to the equation $(\ref{evolution})$, after cancellation we shall get:
\begin{equation}\label{x}
    \dfrac{n\Delta f}{f}+\Delta\ln\det h+g^{i\bar{j}}\dfrac{\partial^2\bar{s_j}}{\partial z^i\partial\bar{\xi}}+\dfrac{1}{g_{\xi\bar{\xi}}}g^{i\bar{j}}\dfrac{\partial\bar{s_j}}{\partial z^i}\dfrac{\partial g_{\xi\bar{\xi}}}{\partial\bar{\xi}}=\dfrac{n}{f}\dfrac{\partial f}{\partial t}+\dfrac{n}{fg_{\xi\bar{\xi}}}\dfrac{\partial ^2f}{\partial\xi\partial\bar{\xi}}.
\end{equation}
Considering the difference:
\begin{equation}\label{y}
\begin{split}
    g^{i\bar{j}}\dfrac{\partial^2\bar{s_j}}{\partial z^i\partial\bar{\xi}}+\dfrac{1}{g_{\xi\bar{\xi}}}g^{i\bar{j}}\dfrac{\partial\bar{s_j}}{\partial z^i}\dfrac{\partial g_{\xi\bar{\xi}}}{\partial\bar{\xi}}-\dfrac{n}{fg_{\xi\bar{\xi}}}\dfrac{\partial ^2f}{\partial\xi\partial\bar{\xi}}
    &=\dfrac{1}{fg_{\xi\bar{\xi}}}\dfrac{\partial}{\partial\bar{\xi}}\left(h^{i\bar{j}}\dfrac{\partial\bar{s_j}}{\partial z^i}g_{\xi\bar{\xi}}-n\dfrac{\partial f}{\partial\xi}\right).
\end{split}
\end{equation}
Now as $g_{i\bar{j}}=fh_{i\bar{j}}+g_{\xi\bar{\xi}}s_i\bar{s_j}$ we get:
\[
    \dfrac{\partial}{\partial\xi}g_{i\bar{j}}=h_{i\bar{j}}\dfrac{\partial f}{\partial\xi}+\bar{s_j}\dfrac{\partial}{\partial\xi}(s_ig_{\xi\bar{\xi}})=h_{i\bar{j}}\dfrac{\partial f}{\partial\xi}+\bar{s_j}\dfrac{\partial}{\partial z^i}g_{\xi\bar{\xi}}.
\]
Here we used the facts that $g_{i\bar{\xi}}=s_ig_{\xi\bar{\xi}}$ and $\dfrac{\partial}{\partial\xi}\bar{s_j}=0$. On the other hands:
\[
    \dfrac{\partial}{\partial\xi}g_{i\bar{j}}=\dfrac{\partial}{\partial z^i}g_{\xi\bar{j}}=\dfrac{\partial}{\partial z^i}(\bar{s_j}g_{\xi\bar{\xi}})=\dfrac{\partial\bar{s_j}}{\partial z^i}g_{\xi\bar{\xi}}+\bar{s_j}\dfrac{\partial}{\partial z^i}g_{\xi\bar{\xi}}.
\]
Comparing the two equations above, we get:
\begin{equation}\label{kahler property}
    h_{i\bar{j}}\dfrac{\partial f}{\partial\xi}=\dfrac{\partial\bar{s_j}}{\partial z^i}g_{\xi\bar{\xi}}.
\end{equation}
From this it follows that:
\begin{equation}\label{z}
    h^{i\bar{j}}\dfrac{\partial\bar{s_j}}{\partial z^i}g_{\xi\bar{\xi}}-n\dfrac{\partial f}{\partial\xi}=h^{i\bar{j}}h_{i\bar{j}}\dfrac{\partial f}{\partial\xi}-n\dfrac{\partial f}{\partial \xi}=0.
\end{equation}
Combining all $(\ref{x}),(\ref{y})$ and $(\ref{z})$ we come up with the following simple evolution equation of $f$:
\[
    \dfrac{\partial f}{\partial t}=\Delta f+\dfrac{f\Delta\ln\det h}{n}=\Delta f-\dfrac{1}{n}h^{i\bar{j}}R^h_{i\bar{j}}=\Delta f-\dfrac{R^h}{n}.
\]
Where the term $R^h_{i\bar{j}}$ and $R^h$ are the Ricci curvature and scalar curvature of $(B,h)$ at the point $\pi(p)$. Notice that if $(B,h)$ is K$\ddot{\text{a}}$hler-Einstein, then the last term is constant. 
\end{proof}
\begin{Prop}
$(B,h)$ must be K$\ddot{\text{a}}$hler-Einstein manifold. 
\end{Prop}
\begin{proof}
From previous calculations, we have that along the flow:
\[
    R_{i\bar{j}}=-\dfrac{\partial f}{\partial t}h_{i\bar{j}}+\dfrac{R_{i\bar{\xi}}R_{\xi\bar{j}}}{R_{\xi\bar{\xi}}}=\left(-\Delta f+\dfrac{R^h}{n}\right)h_{i\bar{j}}+R_{\xi\bar{\xi}}s_i\bar{s_j}.
\]
Now using the fact that $R_{A\bar{B}}=-\dfrac{\partial^2}{\partial A\partial\bar{B}}\ln\det g$, we get the following identity:
\[
    \dfrac{\partial^2}{\partial z^i\partial\bar{z}^j}(\ln g_{\xi\bar{\xi}}+n\ln f+\ln\det h)=\left(\Delta f-\dfrac{R^h}{n}\right)h_{i\bar{j}}+s_i\bar{s_j}\dfrac{\partial^2}{\partial\xi\partial\bar{\xi}}(\ln g_{\xi\bar{\xi}}+n\ln f).
\]
Expanding both sides and notice that since $\nabla f$ is vertical, its Laplacian can be calculated as follows:
\[\Delta f=\dfrac{1}{g_{\xi\bar{\xi}}}\left(\dfrac{n}{f}\dfrac{\partial f}{\partial\xi}\dfrac{\partial f}{\partial \bar{\xi}}+\dfrac{\partial ^2f}{\partial\xi\partial\bar{\xi}}\right).\]
Together with $(\ref{kahler property})$ after cancellation we have:
\[
    \dfrac{R^h}{n}h_{i\bar{j}}=-\dfrac{\partial^2}{\partial z^i\partial \bar{z}^j}\ln\det h=R^h_{i\bar{j}}
\]
This holds for all $1\leq i,j\leq n$ if and only if $(B,h)$ is K$\ddot{\text{a}}$hler-Einstein manifold. 
\end{proof}
\subsection{Sufficient conditions}
\begin{proof}
We still work in local submersion coordinates. At time $t=0$, since $\pi$ is horizontally homothetic conformal submersion with totally geodesic fibers, the local representation of $g_0$ must be in the form given in $(\ref{matrix g})$. \\
Together with condition that $\dfrac{\partial s_i}{\partial\bar{\xi}}=0$ and $\nabla f$ is vertical, we can calculate the matrix representation of Ricci tensor as follows:
\begin{equation}\label{ricci curvature 1}
\text{Ric}_{A\bar{B}}=
\left[
\begin{array}{c|c}
\left(-\Delta f+\dfrac{R_B}{n}\right)h_{i\bar{j}}+ \dfrac{R_{i\bar{\xi}}R_{\xi\bar{j}}}{R_{\xi\bar{\xi}}}& s_iR_{\xi\bar{\xi}} \\
\hline
\overline{s_j}R_{\xi\bar{\xi}} & R_{\xi\bar{\xi}}
\end{array}
\right]
\end{equation}
Now consider the following evolution of K$\ddot{\text{a}}$hler metrics:
\begin{equation}
g(t)=
\left[
\begin{array}{c|c}
f_th_{i\bar{j}}+ s_i\bar{s_j}g_{\xi\bar{\xi}}(t)& s_ig_{\xi\bar{\xi}}(t) \\
\hline
\overline{s_j}g_{\xi\bar{\xi}}(t) & g_{\xi\bar{\xi}}(t)
\end{array}
\right]
\end{equation}
Same calculations as $\textbf{Lemma \ref{ratio}}$ show that as long as $g_{\xi\bar{\xi}}$ evolves as :
\begin{equation}
    \dfrac{\partial}{\partial t}g_{\xi\bar{\xi}}=-R_{\xi\bar{\xi}}=\dfrac{\partial^2}{\partial\xi\partial\bar{\xi}}\left(\ln g_{\xi\bar{\xi}}+n\ln f_t\right).
\end{equation}
As $f_t$ evolves as in the assumption, the above form will be the unique solution to the K$\ddot{\text{a}}$hler-Ricci flow, with the corresponding local representation of Ricci curvature tensor as in $(\ref{ricci curvature 1})$. This form of $g(t)$ clearly indicate that $\pi: (M,g(t))\to (B,h)$ is a conformal submersion with dilation $f_t$. Furthermore, by the assumption on the vertical of $f$ and the preservation of geodesic fibers, the preservation of horizontally homothetic conformal submersion follows. 
\end{proof}

From the evolution equation, we obtain the following uniform bound for $f_t$ and an upper bound for the norm of gradient of $f$ along the flow. 
\begin{Cor}\label{f bounded}
Under the flow, there exist constants $C_1, C_2 >0$ independent of $t$ such that on $M\times [0,T)$ we have: 
\begin{equation}
    C_1\leq f \leq C_2
\end{equation}
\end{Cor}
\begin{proof}
Under the assumption $(\ref{initial})$, we have by the Parabolic Schwarz Lemma:
\begin{equation}
    \omega_t\geq C\pi^*\omega_B.
\end{equation}
Thus $tr_{\omega_t}\pi^*\omega_B=\dfrac{n}{f_t}\leq C$, this establishes the lower bound of $f_t$. (Noticing that we can use maximum principle for heat type evolution equation that $f$ satisfies, however since we do not know about the bound of $R^h$, the lower bound obtained this way is not guaranteed to be positive). 

The upper bound follows from applying maximum principle for $f$ under the evolution $(\ref{evolution of dilation})$.
\end{proof}
The following is a general fact for function evolving under heat-type equation along Ricci flow:
\begin{Prop}\label{bound of gradient}
Suppose that $f$ is a solution to the heat equation with respect to metrics $g(t)$ evolving by Ricci flow:
\begin{equation}
\dfrac{\partial f}{\partial t}=\Delta_{g(t)}f+\phi.
\end{equation}
with $\phi:M\to R$ is time independent function such that $||\nabla\phi||_{g(t)}\leq C$ for a uniform constant $C$. 
 Then there exists a uniform constant $C'$ such that:
 \begin{equation}
     ||\nabla f||_{g(t)}\leq C'
 \end{equation}
 for all $t\in [0,T)$. 
\end{Prop}
\begin{proof}
A standard calculation from \cite{6} yields the following identity:
\begin{equation}
    \dfrac{\partial}{\partial t}||\nabla f||^2=\Delta||\nabla f||^2-2||\nabla\nabla f||^2+2g(\nabla f,\nabla \phi).
\end{equation}
By Schwarz inequality we have:
\begin{equation}
    g(\nabla f,\nabla\phi)\leq ||\nabla f||||\nabla\phi||
\end{equation}
and by the bound of $||\nabla\phi||$, it follows:
\begin{equation}
    \dfrac{\partial}{\partial t}||\nabla f||^2\leq\Delta||\nabla f||^2+K||\nabla f||.
\end{equation}
Denote by $u=||\nabla f||^2$ and by AM-GM inequality: $2||\nabla f||\leq ||\nabla f||^2+1$ we get:
\begin{equation}
    \dfrac{\partial}{\partial t}u\leq \Delta u+K(u+1).
\end{equation}
Since $F(u)=K(u+1)$ is Lipschitz, by a standard application of maximum principal we have:
\begin{equation}
    u\leq e^{kt}-1+c_0,
\end{equation}
where $c_0=\max_{M\times 0}u$. Since time existence of the flow is finite, the bound of $||\nabla f||$ follows.
\end{proof}
In our case, $\phi=-\dfrac{R^h}{n}$ and we want to show that $||\nabla \phi||$ is bounded. This is obvious since $\dfrac{\partial R^h}{\partial \xi}=\dfrac{\partial R^h}{\partial \bar{\xi}}=0$ we have:
\begin{equation}
    ||\nabla R^h||^2_{g(t)}=g^{i\bar{j}}\dfrac{\partial R^h}{\partial z^i}\dfrac{\partial R^h}{\partial\bar{z^j}}=\dfrac{1}{f}h^{i\bar{j}}\dfrac{\partial R^h}{\partial z^i}\dfrac{\partial R^h}{\partial\bar{z^j}}=\dfrac{1}{f}||\nabla R^h||^2_{h}
\end{equation}
which is bounded since $f$ is uniformly bounded and $||\nabla R^h||_{h}$ is time-independent. 

\section{Singularity Analysis}
In this section we will prove $\textbf{Theorem \ref{main theorem}}$ in a slightly more general situation. We shall loosen the restriction on the dimension of $M$ and $B$ (other conditions remain unchanged). In this case, since we do not know if the vanishing of $T$ tensor is preserved under the flow, we will take it for granted and assume that $T=0$ all the time (this condition is automatically guaranteed in dimension one fibers case, as shown in section $3$). In addition, we prescribe a uniform bound of dilation $f$ and a upper bound on its gradient's norm. (these are results we derived in dimension one fibers case). 

The Cheeger-Gromov pointed convergence follows from the uniform bounded assumption of Riemannian tensors together with Hamilton's compactness and Perelman's no local collapsing theorems. Hence we are only interested in the splitting of universal cover of $(M_{\infty}, g_{\infty}(t))$. 

The following consequence of $\textit{holonomy splitting theorem}$ will tell us how to proceed:
\begin{Lemma}\label{holonomy}
Suppose that $(M,g)$ is a Riemannian manifold and that $TM=\mathcal{H}M\oplus\mathcal{V}M$ is an orthogonal decomposition of the tangent bundle of $M$.The O'Neill tensors $A$ and $T$ are defined for this decomposition as followed:
\[
    T_EF=\mathcal{V}\nabla_{\mathcal{V}E}\mathcal{H}F+\mathcal{H}\nabla_{\mathcal{V}E}\mathcal{V}F, \]
   \[ A_EF=\mathcal{V}\nabla_{\mathcal{H}E}\mathcal{H}F+\mathcal{H}\nabla_{\mathcal{H}E}\mathcal{V}F.\]
If both of the tensor vanish, then locally $(M,g)$ is isometric to a product $(N_1\times N_2, g_1\oplus g_2)$, such that $TN_1=\mathcal{H}M$ and $TN_2=\mathcal{V}M$. 
\end{Lemma}
\begin{proof}
According to $\textit{Holonomy splitting theorem}$, it suffices to prove that for vector field $H\in \mathcal{H}M$, we have $\nabla_XH\in \mathcal{H}M$ for any $X\in TM$. Since $X$ can be decomposed as $X_1+X_2$ with $X_2\in \mathcal{H}M$ and $X_2\in \mathcal{V}M$ we have:
\[\nabla_XH=\nabla_{X_1}H+\nabla_{X_2}H.\]
Now $0=T_{X_1}H=\mathcal{V}\nabla_{X_1}H$ since $\mathcal{V}H=0$, thus $\nabla_{X_1}H\in \mathcal{H}M$. Similarly $\nabla_{X_2}H\in\mathcal{H}M$ by the vanishing of tensor $A$, claiming that $\nabla_XH\in \mathcal{H}M$. \\
Suppose that $H_x\in \mathcal{H}M|_x$ and $H(s)$ is the parallel translation of $H_x$ along $\gamma(s)$. We claim that $H(s)\in\mathcal{H}M|_{\gamma(s)}$. Decompose $H(s)=\mathcal{H}H(s)+\mathcal{V}H(s)$. Then
\[0=\nabla_{\gamma'(s)}H(s)=\nabla_{\gamma'(s)}\mathcal{H}H(s)+\nabla_{\gamma'(s)}\mathcal{V}H(s).\]
Since $\nabla_{\gamma'(s)}\mathcal{H}H(s)\in \mathcal{H}M$ by above argument, it follows that $\nabla_{\gamma'(s)}\mathcal{V}H(s)$ is in $\mathcal{H}M$. Therefore:
\[\dfrac{d}{ds}||\mathcal{V}H(s)||^2=2g(\nabla_{\gamma'(s)}\mathcal{V}H(s),\mathcal{V}H(s))=0.\]
Thus, $\mathcal{V}H(s)$ vanishes and then $H(s)$ remains inside $\mathcal{H}M$ along parallel transport. Apply splitting theorem we prove the lemma.
\end{proof}
To apply this theorem, we will decompose the tangent bundle of $(M_\infty, g_\infty(t))$ into orthogonal sub-bundles, then show that the corresponding $A$ and $T$ tensors vanish. By the assumption that the the conformal submersion condition is preserved along the flow, the natural decomposition of $(M_\infty, g_\infty(t))$ is just the limiting of horizontal and vertical distribution $\mathcal{H}_{g_i(t)}M$ and $\mathcal{V}_{g_i(t)}M$ as $i$ goes to $\infty$. Noticing that the horizontal distribution $\mathcal{V}M$ is fixed for all time because the conformal map $\pi$ is unchanged. Therefore, $\mathcal{H}_\infty M$ is the orthogonal complement of $\mathcal{V}M=\ker(\pi_*)$ with respect to $g_\infty(t)$.

Since we already proved that $T$ tensor vanishes along the flow, its limit on $M_\infty$ also vanishes and thus we are left to prove the vanishing of $A$ tensor on the Cheeger-Gromov limit. 

First, we will calculate the norm of $A$ tensor, with the help of $\textbf{Lemma \ref{A_1}}$ and $\textbf{Lemma \ref{vertical of Lie}}$. Suppose that $X_1,\ldots, X_m$ and $V_1, \ldots, V_k$ are the orthonormal bases for the horizontal and vertical distributions, respectively, then we have:
\begin{equation}\label{norm A}
    ||A||^2=\sum_{i,j}||A_{X_i}X_j||^2+\sum_{p,q}||A_{X_p}V_q||^2.
\end{equation}
We are left to calculate $||A_{X_p}V_q||^2$. This can be done as follows:
\[
    A_{X_p}V_q=\mathcal{H}\nabla_{X_p}V_q=g(\nabla_{X_p}V_q,X_l)X_l=-g(V_q,\nabla_{X_p}X_l)X_l.
\]
Thus
\[
    ||A_{X_p}V_q||^2=\sum_{l=1}^{m}(g(V_q,\nabla_{X_p}X_l))^2.
\]
Thus if we take the sum of these expressions over $p,q$, we get that:
\[
    \sum_{p,q}||A_{X_p}V_q||^2=\sum_{p,q,l}(g(V_q,\nabla_{X_p}X_l))^2=\sum_{p,l}||\mathcal{V}\nabla_{X_p}X_l||^2=\sum_{p,l}||A_{X_p}X_l||^2.
\]
Therefore:
\begin{equation}\label{A}
    ||A||^2=2\sum_{i,j}||A_{X_i}X_j||^2.
\end{equation}
Combining with $\textbf{Lemma}$ $ (\ref{vertical of Lie})$, we get the following important estimate for the norm of $A$ tensor along the flow:
\begin{Prop}
There exists a constant $C$ which depends only on $M, B$ and the initial metrics such that:
\begin{equation}\label{estimate A}
    ||A||_{g(t)}^2\leq C||\nabla \ln f_t||_{g(t)}^2=C\dfrac{||\nabla f_t||_{g(t)}^2}{f_t^2}.
\end{equation}
\end{Prop}
With respect to the re-scaling of the metrics: 
\[
    g_i(t)=K_ig(t_i+K_i^{-1}t),
\]
the map $\pi: (M, g_i(t))\rightarrow (B,h)$ is still a horizontally homothetic conformal submersion, with the dilation re-scaled to $K_if_{t_i+K_i^{-1}t}$. Under this, the norm square of $A$ re-scales as follows:
\begin{equation}
    ||A||_{g_i(t)}^2=\dfrac{1}{K_i}||A||_{g(t_i+K_i^{-1}t)}.
\end{equation}
While the quantity $||\nabla \ln f||^2$ does not re-scale. There for we have:
\begin{equation}\label{A_3}
    ||A||_{g_i(t)}^2\leq\dfrac{C}{K_i}\dfrac{||\nabla f||_{g(t_i+K_i^{-1}t)}^2}{f^2}.
\end{equation}
We want the right hand side of $(\ref{A_3})$ to tend to $0$ as $i \rightarrow \infty$. By the way $K_i$ is chosen, we know that $K_i\rightarrow \infty$, and also $f$ is uniformly bounded by $\textbf{Corollary \ref{f bounded}}$. Furthermore, it follows from $\textbf{Proposition \ref{bound of gradient}}$ that the norm of the gradient of $f$ is bounded from above during the flow. Therefore the the norm of O'Neill $A$ tensor tends to $0$ as $i\to \infty$.

As a result, with respect to the limiting decomposition on $(M_\infty, g_\infty(t))$, the O'Neill tensor $A_{g_\infty(t)}$ vanishes and thus the splitting of the universal cover of $(M_\infty)$ into $N_1\times N_2$ follows from $\textbf{Lemma \ref{holonomy}}$. Moreover, each $N_i$ is K$\ddot{\text{a}}$hler since their holonomy groups are subgroups of the unitary groups. Hence, the K$\ddot{\text{a}}$hler-Ricci solution $g_\infty(t)$ splits isometrically into $h_1(t)\oplus h_2(t)$. \\
Suppose that $TN_1=\mathcal{H}M_\infty$ is the limit of $\mathcal{H}_{g_i(t)}M$ when $i\rightarrow \infty$. We shall prove that $N_1$ is flat by calculating its sectional curvatures of the underlying real structure.

The following identity from \cite{4} is useful here:
\begin{Lemma}
Suppose that $\pi: (M,g)\rightarrow (B,h)$ is horizontally homothetic conformal submersion of Riemannian manifolds with dilation $f$. Then at any point $p\in M$ and horizontal vectors $x,y \in T_pM$ we have:
\begin{equation}\label{sectional}
    \dfrac{1}{f}\kappa_B(\pi_*x\wedge\pi_*y)=\kappa_M(x\wedge y)+3||A_XY||_g^2+||\nabla\ln\dfrac{1}{f}||_g^2,
\end{equation}
where $\kappa(x\wedge y)$ denotes the sectional curvature of the plane spanned by vectors $x,y$ and $X,Y$ are vector fields in $M$ which extend $x$ and $y$. 
\end{Lemma}
Now suppose that $(x_1,\ldots,x_k,y_n\ldots,y_k)$ are local submersion coordinates around $p\in M$ such that $\pi(x_1,\ldots y_k)=(x_1,\ldots,x_n)$. Taking $x^i=\mathcal{H}_{g_i(t)}\dfrac{\partial}{\partial x_i}$ and $y^i=\mathcal{H}_{g_i(t)}\dfrac{\partial}{\partial x_j}$ into the equation $(\ref{sectional})$ then take limit as $i\rightarrow 0$. 
Since the metric is being re-scaled as $g_i(t)=K_ig(t_1+K^{-1}t)$, the dilation is re-scaled to $K_if$. As we already proved that $||A||_{g_i(t)}\rightarrow 0$ and the gradient of $f_t$ being bounded, we have:
\begin{equation}
    \kappa_M(x^{\infty}\wedge y^{\infty})=0
\end{equation}
for $x^{\infty},y^{\infty} \in \mathcal{H}_{g_\infty(t)}(M)=TN_1$. Therefore, $(N_1,h_1(t))$ is a flat solution of K$\ddot{\text{a}}$hler-Ricci flow and the $\textbf{Theorem \ref{main theorem}}$ follows. 

\section{Formation of type I singularity in dimension one fibers}
For the case $\dim_{\mathbb{C}}M=n+1, \dim_{\mathbb{C}}B=n$, suppose that $\{x_1,\ldots,x_{2n}, v_1,v_2\}$ is an orthonormal basis of $T_\mathbb{R}M$ and that $x_i, v_j$ are in the horizontal and vertical distribution, respectively. We shall find asymptotic behaviors of the scalar curvature and norm of Riemannian tensor of $M$. The assumptions about conformal submersion in the introduction remain unchanged. 

We shall first analyze the relation between the norm of Riemannian tensor and scalar curvature and claim that as the time approaches singularity, the dominant factors in their expressions are similar. This directly leads to the exclusion of type II. 

The following propositions say that the dominant factors in the expressions of Riemannian tensor as well as scalar curvature come from the contribution of the vertical distribution. This is expected since we knew that under the Cheeger-Gromov convergence, the horizontal distribution becomes flat and hence does not contribute in the calculation of curvatures. Thus, we are left dealing with mixed-type terms.  
\begin{Prop}
Suppose that $\pi: (M,g) \to (B, h)$ is horizontally homothetic conformal submersion of K$\ddot{\text{a}}$hler manifolds with totally geodesic fibers ($T$ tensor vanishes). Then:
\[R(A,B,C,D)=0\]
if among $A,B,C,D$ there are three vectors come from the horizontal distribution and the other is vertical, or vice verse. 
\end{Prop}
\begin{proof}
Firstly if there are three vertical and one horizontal vectors: by Bianchi identity, we just need to prove:
\begin{equation}
    R(V_1,V_2,X, V_3)=g(R(V_1,V_2)V_3,X)=0.
\end{equation}
For $V_i\in \mathcal{V}M$ and $X\in \mathcal{H}M$. We have:
\begin{equation}\label{vertical}
    R(V_1,V_2)V_3=\nabla_{V_1}\nabla_{V_2}V_3-\nabla_{V_2}\nabla_{V_1}V_3-\nabla_{[V_1,V_2]}V_3.
\end{equation}
We will prove that this vector field is vertical. By the vanishing of $T$ tensor, it follows that $\mathcal{H}\nabla_{V_i}V_j=0$ and thus $\nabla_{V_i}V_j$ is vertical. Therefore, $\nabla_{V_i}\nabla_{V_j}V_k$ and also $[V_i,V_j]=\nabla_{V_i}V_j-\nabla_{V_j}V_i$ are vertical. Thus the expression in $(\ref{vertical})$ is vertical.\\
For the second case, we need to prove that $\mathcal{H}R(X,V)Y=0$ for horizontal $X, Y$ and vertical $V$. This uses some computations in $\textbf{Lemma \ref{vertical of Lie}}$. Specifically for any vertical vector $Z$ by the Koszul formula we have:
\begin{equation}\label{1}
    2g(\nabla_VX,Z)=Vg(X,Z)-g([X,Z],V)
\end{equation}
as $[X,V], [Y,V]$ are both vertical. Furthermore from the identity $(\ref{closed omega})$, we deduce that $g([X,Z],V)=-JVg(JX,Z)$ and therefore $(\ref{1})$ becomes:
\begin{equation}\label{2}
    2g(\nabla_VX,Z)=Vg(X,Z)+JVg(JX,Z).
\end{equation}
Now as $X,Z$ are horizontal, we have:
\begin{equation}
    Vg(X,Z)=V(f.h(\pi_*X,\pi_*Z))=h(\pi_*X,\pi_*Z)V(f)=g(\nabla\ln f,V)g(X,Z).
\end{equation}
And similar for $JVg(JX,Z)$ and thus from $(\ref{2})$ we have:
\begin{equation}
    \nabla_VX=g(\nabla\ln f,V)X+g(\nabla\ln f,JV)JX.
\end{equation}
Using this to calculate $\mathcal{H}R(X,V)Y$ we can see that it equals to $0$. The proposition is proved. 
\end{proof}

The following is a direct consequence of $\textbf{Proposition 3.1}$ in \cite{11} is useful in calculating $||Rm||$:
\begin{Prop} Suppose that $\pi: (M,g)\to (B,h)$ is a horizontally homothetic conformal submersion of K$\ddot{\text{a}}$hler manifolds with totally geodesic fiber. Then for all unit horizontal and vertical vector fields $x_j,v_i$ we have:
\begin{equation}
R(v_i,x_j,v_i,x_j)=-\dfrac{1}{2}\left(g(\nabla_{v_i}\nabla\ln f,v_i)+g(\nabla\ln f,U)^2\right)+||A_{x_j}v_i||^2
\end{equation}
\end{Prop}
\begin{proof}
The detailed calculation is given in [7], we only sketch the main ideas here: For basic horizontal vector field $X$ and vertical vector field $U$ we have:
\begin{equation}\label{r}
    R(U,X,U,X)=g(R(U,X)X,U)=g(\nabla_U\nabla_XX,U)-g(\nabla_X\nabla_UX,U)-g(\nabla_{[U,X]}X,U)
\end{equation}
Since $T$ vanishes and $[U,X]\in \mathcal{V}M$ we have $\mathcal{V}\nabla_{[U,X]}X=0$,so the last term of $(\ref{r})$ vanishes.

Also:
\begin{equation}\label{B}
    g(\nabla_U\nabla_XX,U)=g(\nabla_U(\mathcal{H}\nabla_XX+\mathcal{V}\nabla_XX),V)=g(\nabla_U(A_XX),U).
\end{equation}
By $\textbf{Lemma \ref{A_1}}$ we have $A_XX=-\dfrac{1}{2}g(X,X)\nabla\ln f$, as $\pi$ is homothetic.Furthermore:
\begin{equation}\label{C}
    g(\nabla_X\nabla_UX,U)=g(\nabla_X(\mathcal{H}\nabla_UX),U)=-g(\mathcal{H}\nabla_UX,\nabla_XU)=-||A_XU||^2.
\end{equation}
Here the last equality comes from the fact that $0=\mathcal{H}[U,X]=\mathcal{H}\nabla_UX-\mathcal{H}\nabla_XU$. Replace $X,U$ by unit vector fields $v_i,x_j$ in $(\ref{B})$ and $(\ref{C})$ we have our desired identity.
\end{proof}
If we take a sequence of times $t_i \rightarrow T$, in the previous section we have the uniform upper bounds of $||\nabla\ln f||_{g(t)}$ and $||A||_{g(t)}$, the above implies that $R(v_i,x_j,v_i,x_j)=O(1)$ as $i\rightarrow \infty$.
Under the re-scaling procedure: $g_i(t)=K_ig(t_i+K_i^{-1}t)$, we have that: 
\begin{equation}\label{Rm}
K_i^{-2}||Rm||^2_{g(t_i)}=K_i^{-2}4R(u,v,u,v)^2|_{g(t_i)}+K_i^{-2}R(x_i,x_j,x_k,x_l)^2+O(K_i^{-2}).
\end{equation}
The terms $R(x_i,x_j,x_k,x_l)$ where $x_s$ is unit vector on the horizontal distribution are not important because eventually when we take limit as $i\to \infty$, they tend to the components of Riemannian tensor of the horizontal part of $M_\infty$, which has been shown to be flat by $\textbf{Theorem \ref{main theorem}}$ and thus its Riemannian tensor vanishes. 

With respect to the orthonormal basis, the scalar curvature is calculated as:
\begin{equation}
    R_{g(t)}=\sum_{i=1}^{2n}R(x_i,x_i)+R(v_1,v_1)+R(v_2,v_2).
\end{equation}
All terms $R(x_i,x_i)$ will converge to $0$ eventually on $M_{\infty}$, therefore we only need to take care of the last two terms, which are the dominant terms of $R_{g(t)}$:
\begin{equation}
R(v_1,v_1)=\sum_{i=1}^{2n}R(v_1,x_i,v_1,x_i)+R(v_1,v_2,v_2,v_2)=O(1)+R(v_1,v_2,v_1,v_2).
\end{equation}
Similarly:
\begin{equation}
    R(v_2,v_2)=O(1)+R(v_2,v_1,v_2,v_1).
\end{equation}
Since $R(v_1,v_2,v_1,v_2)=R(v_2,v_1,v_2,v_1)$, we have that:
\begin{equation}\label{scalar}
    R_{g(t_i)}=2R(v_1,v_2,v_1,v_2)+O(1)
\end{equation}
Observe that from $(\ref{Rm})$ and $(\ref{scalar})$ we have that essentially the dominant parts of $R_{g(t_i)}$ and $||Rm||_{g(t_i)}$ are the same. We are now in the position to exclude type II singularity:
\begin{Th}\label{exclusion of type II}
Suppose that $\pi: M^{n+1}\rightarrow B^n$ is dimension one fibers submersion and on $M$ we run a K$\ddot{\text{a}}$hler-Ricci flow which satisfies all conditions in the introduction. Then, the singularity cannot be of type II. 
\end{Th}
\begin{proof}
We will prove by contradiction. Suppose that the singularity is indeed of type II, then the standard point-picking is performed as follows. First, we take $t_i\to T$ and let $(x_i,t_i)\in M\times [0,T_i]$ such that:
\begin{equation}
    (T_i-t_i)||Rm||(x_i,t_i)=\max_{M\times [0,T_i]}(T_i-t)||Rm||_{g(t)}.
\end{equation}
The procedure of choosing $K_i$ guarantees that $1\geq K_i^{-2}||Rm||^2_{g(t_i)}$ on $M$ and the equality holds at $x_i$. Now as $i\rightarrow \infty$, by $\textbf{Theorem \ref{main theorem}}$ we know that $M_{\infty}$ splits into $(\mathbb{C}^n\times N_2, ||dz||^2\oplus h_2(t))$. From the asymptotic behaviors of $||Rm||$ and $R$ in $(\ref{Rm})$ and $(\ref{scalar})$ we get that:
\begin{equation}\label{positive of R}
    R_{g_\infty(t)}=R_{N_2}(h_2(t))\leq 1,
\end{equation}
and the equality holds at $t=0$ and $x=x_\infty$. By strong maximal principle, the scalar curvature of every ancient solution must be either identically $0$ or positive the whole time. By $(\ref{positive of R})$, we can conclude that:
\begin{equation}\label{bound of R}
    0<R_{h_2(t)}\leq 1.
\end{equation}

Therefore, $(N_2,h_2(t))$ is an eternal solution which satisfies $(\ref{bound of R})$ and thus by Hamilton's classification of eternal solutions in \cite{12}, $(N_2,h_2(t))$ is a steady gradient soliton. In our initial setting of complex dimensions $1$-fiber submersion, we have $\dim_\mathbb{R}N_2=2$ and thus it is a cigar soliton by \cite{13}. This together with our splitting lemma violates the Perelman's non local collapsing in \cite{14} that the Cheeger-Gromov limit $(M_\infty, g_\infty(t))$ must be $\kappa$-non-collapsed at all scales. Therefore, we ended up reaching a contradiction and thus type II singularity is not possible.
\end{proof}
We can finally give a proof of $\textbf{Theorem \ref{singularity}}$:
\begin{proof}
The convergence in the Cheeger-Gromov sense is guaranteed by Hamilton's compactness and Perelman's non local collapsing theorems. The analysis of this limit proceeds similarly as in our previous proof of the exclusion of type II. Precisely by the choice of $K_i$ and $(x_i,t_i)$ we have:
\begin{equation}
    K_i^{-2}||Rm||_{g(t_i)}\leq 1
\end{equation}
and the equality happens at $(x_i,t_t)$. Taking limit of this as $i\to \infty$ and comparing the dominant term of $||Rm||_{g(t_t)}$ with that of the scalar curvature $R_{g(t_i)}$ we shall also have that:
\begin{equation}
    R_{g_\infty(t)}=R_{N_2}(h_2(t))\leq 1
\end{equation}
Here we used the fact that the Cheeger-Gromov limit splits as: \[(\mathbb{C}^n\times N_2,||dz||^2\oplus h_2(t)).\]
Since the flow on $M$ is of type I singularity, we have that $N_2,h_2(t)$ is a type I eternal solution to Ricci flow on surface. Therefore by Hamilton's classification of ancient $\kappa$-solution, $(N_2,h_2(t))$ must be isometric to the shrinking round sphere. 

Together with $\textbf{Theorem \ref{main theorem}}$, we conclude that under our assumption on conformal submersion and dimension one fiber, the K$\ddot{\text{a}}$hler-Ricci flow is of type I singularity and the universal cover of the limit solution $(M_\infty, g_\infty(t))$ splits isometrically as:
\[(\mathbb{C}^n\times \mathbb{CP}^1, ||dz||^2\oplus\omega_{FS}(t).\]
\end{proof}

\bibliographystyle{alpha}
\bibliography{references}

\begin{thebibliography}{Ham93b}

\bibitem[CLN06]{6}
Bennett Chow, Peng Lu, and Lei Ni.
\newblock {\em Hamilton's Ricci flow}, volume~77.
\newblock American Mathematical Soc., 2006.

\bibitem[Fon14]{1}
Frederick Tsz-Ho Fong.
\newblock K\"{a}hler-{R}icci flow on projective bundles over
  {K}\"{a}hler-{E}instein manifolds.
\newblock {\em Trans. Amer. Math. Soc.}, 366(2):563--589, 2014.

\bibitem[Gud92]{7}
Sigmundur Gudmundsson.
\newblock {\em The geometry of harmonic morphisms}.
\newblock PhD thesis, University of Leeds (Department of Pure Mathematics),
  1992.

\bibitem[Ham93a]{13}
Richard Hamilton.
\newblock The formations of singularities in the ricci flow.
\newblock {\em Surveys in differential geometry}, 2(1):7--136, 1993.

\bibitem[Ham93b]{12}
Richard~S Hamilton.
\newblock Eternal solutions to the ricci flow.
\newblock {\em Journal of Differential Geometry}, 38(1):1--11, 1993.

\bibitem[Okr98]{15}
Sergey~Ivanovich Okrut.
\newblock The conformal submersions of k{\"a}hlerian manifolds. i.
\newblock {\em Zhurnal Matematicheskoi Fiziki, Analiza, Geometrii [Journal of
  Mathematical Physics, Analysis, Geometry}, 5(3):228--249, 1998.

\bibitem[OW07]{4}
Ye-Lin Ou and Frederick Wilhelm.
\newblock Horizontally homothetic submersions and nonnegative curvature.
\newblock {\em Indiana University Mathematics Journal}, 56(1):243--261, 2007.

\bibitem[Per02]{14}
Grisha Perelman.
\newblock The entropy formula for the ricci flow and its geometric
  applications.
\newblock {\em arXiv preprint math/0211159}, 2002.

\bibitem[Son15]{10}
Jian Song.
\newblock Some type $\mathrm{I}$ solutions of ricci flow with rotational
  symmetry.
\newblock {\em International Mathematics Research Notices},
  2015(16):7365--7381, 2015.

\bibitem[SSW13]{5}
J.~{Song}, G.~{Szekelyhidi}, and B.~{Weinkove}.
\newblock The k{\"a}hler-ricci flow on projective bundles.
\newblock {\em International Mathematics Research Notices}, 2013(2):243--257,
  2013.

\bibitem[SW11]{2}
Jian Song and Ben Weinkove.
\newblock The k{\"a}hler--ricci flow on hirzebruch surfaces.
\newblock 2011.

\bibitem[TZ06]{8}
Gang Tian and Zhou Zhang.
\newblock On the k{\"a}hler-ricci flow on projective manifolds of general type.
\newblock {\em Chinese Annals of Mathematics, Series B}, 27(2):179--192, 2006.

\bibitem[Vil70]{3}
Jaak Vilms.
\newblock Totally geodesic maps.
\newblock {\em Journal of differential geometry}, 4(1):73--79, 1970.

\bibitem[Zaw14]{11}
Tomasz Zawadzki.
\newblock Existence conditions for conformal submersions with totally umbilical
  fibers.
\newblock {\em Differential Geometry and its Applications}, 35:69--85, 2014.

\bibitem[Zhu07]{9}
Xiaohua Zhu.
\newblock K{\"a}hler-ricci flow on a toric manifold with positive first chern
  class.
\newblock {\em arXiv preprint math/0703486}, 2007.

\end{thebibliography}
\nocite{*}
\Addresses
\end{document}